\documentclass[reqno,12pt]{amsart}
\usepackage[T1]{fontenc}
\usepackage{enumerate}
\theoremstyle{plain}
\newtheorem{thm}{Theorem}

\begin{document}
\title{On competitive discrete systems in the plane. I. Invariant manifolds.}

\author[Gabriel Lugo]{Gabriel Lugo}
\address{Department of Mathematics, University of Rhode Island,Kingston, RI 02881-0816, USA;}
\email{glugo@math.uri.edu}
\author[Frank J. Palladino]{Frank J. Palladino}
\address{Department of Mathematics, University of Rhode Island,Kingston, RI 02881-0816, USA;}
\email{frank@math.uri.edu}
\date{October 19, 2011}

\begin{abstract}
\noindent Let $T$ be a $C^{1}$ competitive map on a rectangular region $R\subset \mathbb{R}^{2}$. The main results of this paper give conditions which guarantee the existence of an invariant curve $C$, which is the graph of a continuous increasing function, emanating from a fixed point $\bar{z}$. 
We show that $C$ is a subset of the basin of attraction of $\bar{z}$ and that the set consisting of the endpoints of the curve $C$ in the interior of $R$ is forward invariant. The main results can be used to give an accurate picture of the basins of attraction for many competitive maps.\par
We then apply the main results of this paper along with other techniques to determine a near complete picture of the qualitative behavior for the following two rational systems in the plane.
 $$x_{n+1}=\frac{\alpha_{1}}{A_{1}+y_{n}},\quad y_{n+1}=\frac{\gamma_{2}y_{n}}{x_{n}},\quad n=0,1,\dots,$$
with $\alpha_1,A_{1},\gamma_{2}>0$ and arbitrary nonnegative initial conditions so that the denominator is never zero.
$$x_{n+1}=\frac{\alpha_{1}}{A_{1}+y_{n}},\quad y_{n+1}=\frac{y_{n}}{A_{2}+x_{n}},\quad n=0,1,\dots,$$
with $\alpha_1,A_{1},A_{2}>0$ and arbitrary nonnegative initial conditions.
\end{abstract}
\maketitle

\section{Introduction}
Let $R$ be a subset of $\mathbb{R}^{2}$ with nonempty interior, and let $T:R\rightarrow R$ be a map. Set $T(x,y)=(f(x,y),g(x,y))$. The map $T$ is \emph{competitive} if $f(x,y)$ is nondecreasing in $x$ and nonincreasing in $y$, and $g(x,y)$ is nonincreasing in $x$ and nondecreasing in $y$. 
If $T$ is competitive, the associated system of difference equations,
\begin{equation}
x_{n+1}=f(x_{n},y_{n}),\quad y_{n+1}=g(x_{n},y_{n}),\quad n=0,1,\dots,\quad (x_{0},y_{0})\in R,
\end{equation}
is said to be competitive. The map $T$ and associated system of difference equations are said to be \emph{strongly competitive} if the adjectives nondecreasing and nonincreasing in the prior description are replaced by strictly increasing and strictly decreasing respectively.\par
Competitive systems of the form (1) have been studied by many authors, such as Clark, Kulenovi\'c, and Selgrade \cite{cks}, Dancer and Hess \cite{dh}, Franke and Yakubu \cite{fy1}-\cite{fy3}, Hess and Pol\'a\v{c}ik \cite{hp}, Hirsch and Smith \cite{hsbook}, Kulenovi\'c and Merino \cite{kmbifurcation}-\cite{kminvariantmanifolds}, Smith \cite{sm1}-\cite{sm5}, and others.
The term \emph{competitive} was introduced by Hirsch, see \cite{h}, for systems of autonomous differential equations. Many mathematical models in the biological sciences may be classified as competitive, see \cite{clch}, \cite{dms}, \cite{hs}, \cite{lm} and \cite{s}. Consideration of the Poincar\'e maps of these systems lead naturally to the concept of competitive maps in the discrete case, see \cite{dms}, \cite{sm2} and \cite{sm3}.\par
Recently, there has been a surge of activity in the study of competitive systems in the plane, see \cite{bo}-\cite{bkk}, \cite{gdn}, and \cite{kk}-\cite{kn2}. Several powerful results were obtained which will serve to meet the demands of researchers in mathematical biology and rational difference equations. Of particular interest, are the results of Kulenovi\'c and Merino in \cite{kminvariantmanifolds}. Competitive rational systems in the plane are prone to have several equilibria. Kulenovi\'c and Merino's results in \cite{kminvariantmanifolds} allow for a very detailed qualitative description of the basins of attraction in many cases. The main drawback of Kulenovi\'c and Merino's results in \cite{kminvariantmanifolds} is that the system is required to be strongly competitive so that their results may apply.
So that the reader understands the impact of the requirement that the system be strongly competitive, let us consider the results in the context of the general linear fractional system in the plane.
\begin{equation}
x_{n+1}=\frac{\alpha_{1}+\beta_{1}x_{n}+\gamma_{1}y_{n}}{A_{1}+B_{1}x_{n}+C_{1}y_{n}},\quad y_{n+1}=\frac{\alpha_{2}+\beta_{2}x_{n}+\gamma_{2}y_{n}}{A_{2}+B_{2}x_{n}+C_{2}y_{n}},\quad n=0,1,2,\dots ,
\end{equation}
with all parameters and initial conditions nonnegative so that division by zero is avoided. System (2) has $49^{2}=2,401$ special cases depending on which parameters are chosen to be positive or zero, of which $17^{2}=289$ are competitive. A numbering system was introduced in \cite{cklm} to keep track of the special cases of System (2). 
Using this numbering system, the competitive cases are obtained by pairing any two numbers in the following list: 1, 3, 4, 5, 6, 9, 11, 13, 14, 15, 19, 21, 24, 27, 29, 38 and 42. However, any case which is obtained by pairing any two numbers in the following list 1, 3, 4, 5, 9, 11, 13, 19 and 24 is clearly reducible to a Riccati equation or a linear equation. So, there are $289-81=208$ nontrivial competitive cases.
The assumption (1) from the forthcoming Theorem 1 holds for all competitive cases that are obtained by pairing any two numbers in the following list: 4, 6, 13, 14, 15, 19, 21, 27, 29, 38 and 42. There are $11^{2}=121$ such cases of which $9$ cases are trivial. The assumption (2) from the forthcoming Theorem 1 holds for all competitive cases that are obtained by pairing any two numbers in the following list: 3, 6, 11, 14, 15, 21, 24, 27, 29, 38 and 42. There are $11^{2}=121$ such cases of which $9$ cases are trivial. For the purposes of comparison, note that, using the numbering system in \cite{cklm}, the strongly competitive cases are the cases obtained by pairing any two numbers in the following list: 6, 14, 15, 21, 27, 29, 38 and 42. So, there are $8^{2}=64$ special cases of System (2) that are strongly competitive.

\par
This paper is the first in a series of papers which will address the nontrivial competitive special cases of the following rational system in the plane, labeled system \#11. 
$$x_{n+1}=\frac{\alpha_{1}}{A_{1}+y_{n}},\quad y_{n+1}=\frac{\alpha_{2}+\beta_{2}x_{n}+\gamma_{2}y_{n}}{A_{2}+B_{2}x_{n}+C_{2}y_{n}},\quad n=0,1,2,\dots ,$$
with $\alpha_{1},A_{1}>0$ and $\alpha_{2}, \beta_{2}, \gamma_{2}, A_{2}, B_{2}, C_{2}\geq 0$ so that $\alpha_{2}+\beta_{2}+\gamma_{2}>0$ and $A_{2}+B_{2}+C_{2}>0$ and with nonnegative initial conditions $x_{0}$ and $y_{0}$ so that the denominator is never zero.
It is immediately apparent that no special case of system \#11 is strongly competitive anywhere. However, the tools given by Kulenovi\'c and Merino are essential tools which would allow us to give a near complete description of the qualitative behavior in many of the cases we study. Therefore, it is highly desirable for us to find analogous results for the cases we study. 
To this end, we generalize Kulenovi\'c and Merino's theory so that the new theory may be applied to the systems that we study in this series of papers on the nontrivial competitive special cases of system \#11. Mathematical biologists may also find this generalization useful.
The generalization of Kulenovi\'c and Merino's results has the broadest impact of all the results in the series and will be used throughout the series. In this sense, these results can be considered the main results of the series.\par The nontrivial competitive special cases of system \#11 are the cases numbered $(11,6)$, $(11,14)$, $(11,15)$, $(11,21)$, $(11,27)$, $(11,29)$, $(11,38)$ and $(11,42)$, in the numbering system developed in \cite{cklm}. Our goal in this article is to determine a complete picture of the qualitative behavior for the first two cases in the above list to the best of our ability.\par  
The first system, numbered $(11,6)$ in the numbering system developed in \cite{cklm}, is the system
$$x_{n+1}=\frac{\alpha_{1}}{A_{1}+y_{n}},\quad y_{n+1}=\frac{\gamma_{2}y_{n}}{x_{n}},\quad n=0,1,\dots,$$
with $\alpha_1,A_{1},\gamma_{2}>0$ and arbitrary nonnegative initial conditions so that the denominator is never zero.\par
The second system, numbered $(11,14)$ in the same numbering system, is the system
$$x_{n+1}=\frac{\alpha_{1}}{A_{1}+y_{n}},\quad y_{n+1}=\frac{\gamma_{2}y_{n}}{A_{2}+x_{n}},\quad n=0,1,\dots,$$
with $\alpha_1,A_{1},A_{2},\gamma_{2}>0$ and arbitrary nonnegative initial conditions. Notice that the change of variables $y^{*}_{n}=y_{n}$ and $x^{*}_{n}=\frac{x_{n}}{\gamma_{2}}$ and relabeling of parameters for the sake of notation reduces the second system to
$$x_{n+1}=\frac{\alpha_{1}}{A_{1}+y_{n}},\quad y_{n+1}=\frac{y_{n}}{A_{2}+x_{n}},\quad n=0,1,\dots,$$
with $\alpha_1,A_{1},A_{2}>0$ and arbitrary nonnegative initial conditions.\par
This article is organized as follows. In Section 2, we present some preliminary results and definitions for competitive maps which we will need in the remainder of the article. In Section 3, we present our generalization of Kulenovi\'c and Merino's theory. In Section 4, we present a near complete description of the qualitative behavior of the system $(11,6)$ using in part the results from Section 3.
In Section 5, we present a near complete description of the qualitative behavior of the system $(11,14)$ using in part the results from Section 3.
\section{Preliminary results and definitions for competitive maps}
Competitive maps are order preserving maps. Indeed, denote with $\preceq _{se}$ the \emph{southeast} partial order in the plane whose nonnegative cone is the standard fourth quadrant $\{(x,y)|x\geq 0,\; y\leq 0\}$, that is, $(x,y)\preceq _{se} (u,v)$ if and only if $x\leq u$ and $y\geq v$. Competitive maps in the plane preserve the southeast ordering, that is $T(s)\preceq _{se} T(w)$ whenever $s\preceq _{se} w$. We say that $s$ and $w$ are comparable in the order $\preceq _{se}$ if either $s\preceq _{se} w$ or $w\preceq _{se} s$. \par 
A map $T$ on a nonempty set $R\subset \mathbb{R}^{2}$ is a continuous function $T:R\rightarrow R$. A set $A\subset R$ is invariant for the map $T$ if $T(A)\subset A$. A point $x\in R$ is a fixed point of $T$ if $T(x)=x$, and a minimal period-two point if $T^{2}(x)=x$ and $T(x)\neq x$. The basin of attraction of a fixed point $x$ is the set of all $y$ such that $T^{n}(y)\rightarrow x$. A fixed point $x$ is a global attractor of a set $A$ if $A$ is a subset of the basin of attraction of $x$.\par
Given $(x,y)\in\mathbb{R}^{2}$, define $Q_{1}((x,y))=\{(u,v)\in\mathbb{R}^{2}|u\geq x\; and\; v\geq y\}$, $Q_{2}((x,y))=\{(u,v)\in\mathbb{R}^{2}|u\leq x\; and\; v\geq y\}$, $Q_{3}((x,y))=\{(u,v)\in\mathbb{R}^{2}|u\leq x\; and\; v\leq y\}$ 
and $Q_{4}((x,y))=\{(u,v)\in\mathbb{R}^{2}|u\geq x\; and\; v\leq y\}$. Given $(x,y)\in\mathbb{R}^{2}$ and a competitive map $T$ both $Q_{2}((x,y))$ and $Q_{4}((x,y))$ are invariant.\par

\section{A generalization of prior results on Strongly Competitive systems}
In \cite{kminvariantmanifolds}, Kulenovi\'c and Merino develop a powerful theory for dealing with the dynamics of competitive planar maps. We would like to use this theory to give a clear picture of the qualitative behavior for the two rational systems in the plane $(11,6)$ and $(11,14)$. However, these rational systems are not strongly competitive, so Kulenovi\'c and Merino's theory does not apply. In this section, we adjust the theorems and proofs of Kulenovi\'c and Merino in \cite{kminvariantmanifolds} to allow for cases which are not necessarily strongly competitive.\par
The idea behind the proof of the generalized version is the same idea as was used in the proof given by Kulenovi\'c and Merino. The structure of the argument also has many similarities and is identical in some places. However, we feel that it is necessary to present the new argument in its entirety in the case of Theorem 1 as there are many small changes throughout the proof of the generalized version. On the other hand, for Theorems 2, 3 and 4 large parts of the proof are identical to the proof presented in \cite{kminvariantmanifolds} and we only point out the differences between proofs where such differences arise.
In the following four theorems, we generalize Theorems 1 through 4 of \cite{kminvariantmanifolds}. Particularly, we relax the assumption that the map must be strongly competitive.
\begin{thm}
Let $T$ be a competitive map on a rectangular region $R\subset \mathbb{R}^{2}$ where
$$T\left(\begin{array}{cc}
          x\\
y\\
         \end{array}
\right)= \left(\begin{array}{cc}
f(x,y)\\
g(x,y)\\
         \end{array}
\right).$$ Let $\bar{z}\in R$ be a fixed point of $T$ such that $\Delta := R\cap int\left( Q_{1}(\bar{z})\cup Q_{3}(\bar{z})\right)$ is nonempty (i.e. $\bar{z}$ is not the NW or SE vertex of $R$). Further, assume that $T$ is $C^{1}$ and one of the following holds:
\begin{enumerate}
\item $\frac{\partial f}{\partial x}|_{(s,t)}>0$ and $\frac{\partial g}{\partial y}|_{(s,t)}>0$ for all $(s,t)\in \Delta$.
\item $\frac{\partial f}{\partial y}|_{(s,t)}<0$ and $\frac{\partial g}{\partial x}|_{(s,t)}<0$ for all $(s,t)\in \Delta$.
\end{enumerate}
 Suppose that the  Jacobian $J_{T}(\bar{z})$ of $T$ at $\bar{z}$ has real eigenvalues $\lambda , \mu$ such that $0<|\lambda|<\mu$, where $|\lambda|<1$, and the eigenspace $E^{\lambda}$ associated with $\lambda$ is not a coordinate axis.
Then there exists a curve $C\subset R$ through $\bar{z}$ that is invariant and a subset of the basin of attraction of $\bar{z}$, such that $C$ is tangential to the eigenspace $E^{\lambda}$ at $\bar{z}$, and $C$ is the graph of a strictly increasing continuous function of the first coordinate on an interval. 
Moreover, the set consisting of the endpoints of $C$ in the interior of $R$ is forward invariant.
\end{thm}
\begin{proof}
The proof is based on Theorem 1 of \cite{kminvariantmanifolds} but has a slightly different presentation. Take $\bar{z}\in R$. We have assumed that $\Delta$ is nonempty and we have made appropriate assumptions so that Lemma 5.1 in p. 234 of \cite{hartman1964} applies. Using Lemma 5.1 in p. 234 of \cite{hartman1964}, we get that there exists a small neighborhood $V$ of $\bar{z}$ and a locally invariant $C^{1}$ manifold $\hat{C}\subset V$ that is tangential to $E^{\lambda}$ at $\bar{z}$ and such that $T^{n}(x)\rightarrow \bar{z}$ for all $x\in\hat{C}$. 
Since $T$ is competitive, a unit eigenvector $v^{\lambda}$ associated with $\lambda$ may be chosen so that $v^{\lambda}$ has nonnegative entries. By our assumption regarding the structure of the Jacobian at $\bar{z}$, the vector $v^{\lambda}$ has positive entires. If necessary, since $T$ is $C^{1}$, the diameter of $V$ may be taken to be small enough to guarantee that no two points on $\hat{C}$ are comparable in the ordering $\preceq _{se}$. 
This can be done since $T$ is $C^{1}$ and the tangential vector $v^{\lambda}$ has positive entries. Let $C$ be the connected component of the union of all preimages of $\hat{C}$ that contains $\hat{C}$. More precisely, $C$ is the connected component of $\left\{\bigcup^{\infty}_{i=0}T^{-i}(\hat{C})\right\}$ containing $\hat{C}$. The set $C$ consists of noncomparable points in the order $\preceq _{se}$. Indeed, if $v$ and $w$ are two distinct points in $C$ such that $v \preceq _{se} w$, 
then $T^{n}(v)\preceq _{se} T^{n}(w)$ and $T^{n}(v)\neq T^{n}(w)$ for $n\geq 0$ because of either (1) or (2). But for $n$ large enough, both $T^{n}(v)$ and $T^{n}(w)$ belong to $\hat{C}$, which consists of noncomparable points, a contradiction. Hence, $C$ consists of noncomparable points. The projection of $C$ onto the first coordinate is a connected set, thus it is an interval $J\subset \mathbb{R}$. 
Since points on $C$ are noncomparable, $C$ is the graph of a strictly increasing function $f(t)$ of $t\in J$. The projection of $C$ onto the second coordinate is a connected set, thus it is an interval. If there is a jump discontinuity at $t_{0}\in J$, then this contradicts the connectedness of $C$. Thus, $f(t)$ is a continuous function.\par
$C$ is the graph of a continuous strictly increasing function from $J$ to $\mathbb{R}$ and, while we have not yet demonstrated $C$ to be a $C^{1}$ manifold, we know for a fact that it is a continuous curve in the plane. In order to establish the desired result we need only show that the set consisting of the endpoints (or perhaps endpoint) of this curve which lie in the interior of $R$ is invariant. If all of the endpoints of this curve lie in the boundary of $R$, then the proof is complete.
Otherwise, take $(j_{1},f(j_{1}))$ to be a generic endpoint which does not lie in the boundary of $R$. We will assume, without loss of generality, that $j_{1}$ is the right endpoint of $J$. The other case follows similarly. Assume, for the sake of contradiction, that $T((j_{1},f(j_{1})))$ is not an endpoint of $C$, then, by continuity and the definition of $C$, we get that there is some $j_{2}\in J$, which is not an endpoint of $J$, so that $T((j_{1},f(j_{1})))=(j_{2},f(j_{2}))$. Let $R_{1}$ be the smallest rectangle in $R$ containing $C$.
Note that $C$ is a separartix for $R_1$. Moreover, since $j_{2}$ is in the interior of $J$, $(j_{2},f(j_{2}))$ is in the interior of $R_1$. Thus, there is a sufficiently small $\epsilon>0$ so that the ball of radius $\epsilon$ about $(j_{2},f(j_{2}))$, henceforth denoted $B((j_{2},f(j_{2})),\epsilon)$, lies entirely in $R_{1}$.

By continuity of $T$, there exists $\delta >0$ such that $T(B((j_{1},f(j_{1})),\delta))\subset B\left((j_{2},f(j_{2})),\frac{\epsilon}{2}\right)$. 
Consider the line segment $L_{0}$ with endpoints $\left(j_{1},f(j_{1})\pm \frac{\delta}{2} \right)$. 
Since $L_{0}$ is linearly ordered by $\preceq _{se}$, so is $T(L_{0})$, and the points $T\left(\left(j_{1},f(j_{1})+ \frac{\delta}{2} \right)\right)$ and $T\left(\left(j_{1},f(j_{1})- \frac{\delta}{2} \right)\right)$ are on different components of $R_{1}\setminus C$. 
Find $\epsilon_{2}>0$ such that both $B\left(T\left(\left(j_{1},f(j_{1}) + \frac{\delta}{2}\right) \right),\epsilon_{2}\right)$, and $B\left(T\left(\left(j_{1},f(j_{1}) - \frac{\delta}{2}\right) \right),\epsilon_{2}\right)$ are each a subset of a different component of $R_{1}\setminus C$ and are both subsets of $B((j_{2},f(j_{2})),\epsilon)$. Now, choose $\eta >0$ such that both
$$T\left[B\left(\left(j_{1},f(j_{1})+ \frac{\delta}{2}\right),\eta\right)\right]\subset B\left(T\left[\left(j_{1},f(j_{1})+\frac{\delta}{2} \right)\right],\epsilon_{2}\right),$$
and
$$T\left[B\left(\left(j_{1},f(j_{1})- \frac{\delta}{2}\right),\eta\right)\right]\subset B\left(T\left[\left(j_{1},f(j_{1})-\frac{\delta}{2} \right)\right],\epsilon_{2}\right).$$
Now, for each $t\in (0,\eta)$ consider the line segment $L_{t}$ with endpoints $p_{\pm}(t):=\left(j_{1}+t,f(j_{1})\pm \frac{\delta}{2}\right)$.
Notice that for each $t\in (0,\eta)$, $T(p_{+}(t))$ and $T(p_{-}(t))$ belong to different components of $R_{1}\setminus C$ yet both belong to $B\left((j_{2},f(j_{2})),\epsilon\right)$.
For each $t\in (0,\eta)$, the line segment $L_{t}$ is linearly ordered by $ \preceq _{se}$, hence so is $T(L_{t})$. Since $T(p_{+}(t))$ and $T(p_{-}(t))$ both belong to $int(R_{1})$ and $T(L_{t})$ is ordered by $ \preceq _{se}$, $T(L_{t})\subset int(R_{1})$. 
Thus, by continuity of $T$, for each $t\in (0,\eta)$ there exists $x_{t}\in L_{t}$ such that $T(x_{t})\in C$. That is, the function $f(x)$ may be extended to a function $\hat{f}(x)$ defined on an interval $\hat{J}$ that includes $J$ as a proper subset.
The reasoning used to show monotonicity of $f(x)$ gives monotonicity of $\hat{f}(x)$. If there is a jump discontinuity at $t_{0}\in (0,\eta)$, let $y_{-}$ and $y_{+}$ respectively be the left and right (distinct) limits of $\hat{f}(x)$ as $x$ approaches $j_{1}+t_{0}$ respectively. The points $(j_{1}+t_{0},y_{-})$ and $(j_{1}+t_{0},y_{+})$ are comparable in the order $\preceq _{se}$ and because of either (1) or (2), 
$T(j_{1}+t_{0},y_{-})\preceq _{se} T(j_{1}+t_{0},y_{+})$ and $T(j_{1}+t_{0},y_{+})-T(j_{1}+t_{0},y_{-})\neq (0,0)$. Since both $T(j_{1}+t_{0},y_{+})$ and $T(j_{1}+t_{0},y_{-})$ are accumulation points of $C$ and in $int(R_{1})$, they are in $C$. Thus, we obtain that $C$ must have comparable points, a contradiction. Thus $\hat{f}(t)$ is a continuous function. This contradicts the choice of $C$ as a connected component, and we conclude that there can be no such endpoint $(j_{1},f(j_{1}))$.
So, we have narrowed down the possibilities for $T(j_{1},f(j_{1}))$. We now know that $T(j_{1},f(j_{1}))$ must be an endpoint of $C$. If we can now rule out the possibility that $T(j_{1},f(j_{1}))$ is in the boundary of $R$, then we will have shown that the set consisting of the endpoints (or perhaps endpoint) of $C$ which lie in the interior of $R$ is invariant, as the statement of the theorem claims.
However, this possibility is easily taken care of. In fact, we may show that $T(x_{0})\not\in \partial R$ for any $x_{0}\in int(R)$. To see this, consider points $y$ and $z$ in $int(R)$ so that the points $y$, $x_{0}$, and $z$ lie on a line with slope $-1$ containing $x_{0}$, so that $y \preceq _{se} x_{0} \preceq _{se} z$ with all three points distinct. We may take such points since $x_{0}\in int(R)$ by choosing $y$ and $z$ sufficiently close to $x_{0}$. 
Then, by either (1) or (2) and the fact that $T$ is competitive, $T(y) \preceq _{se} T(x_{0}) \preceq _{se} T(z)$ with all three points distinct. More precisely, $T(y)\in int(Q_{2}(T(x_{0})))$ and $T(z)\in int(Q_{4}(T(x_{0})))$, and, in fact, the assumptions (1) and (2) were designed for this purpose. So, if $x_{0}\in\partial R$, then we are forced to conclude that one of the points $T(y)$ or $T(z)$ does not belong to $R$, contradicting the assumption that $R$ is forward invariant under $T$.
So, we have now shown that the set consisting of endpoints of $C$ which are not in the boundary of $R$ is forward invariant under $T$ and the theorem is proved. \par

\end{proof}

\begin{thm}
For the curve $C$ of Theorem 1 to have endpoints in $\partial R$, it is sufficient that at least one of the following conditions is satisfied.
\begin{enumerate}[(i)]
\item The map $T$ has no fixed points nor periodic points of minimal period two in $\Delta$.
\item The map $T$ has no fixed points in $\Delta$, $det J_{T}\left(\bar{z}\right)>0$, and $T(x)=\bar{z}$ has no solutions $x\in\Delta$.
\item The map $T$ has no fixed points in $\Delta$, $det J_{T}\left(\bar{z}\right)<0$, and $T(x)=\bar{z}$ has no solutions $x\in\Delta$.
\end{enumerate}
\end{thm}
\begin{proof}
We have shown, in Theorem 1, that the set consisting of endpoints of $C$ which are not in the boundary of $R$ is forward invariant under $T$. This set has cardinality at most two and so either it is empty, it has a fixed point, or it has a minimal period 2 point. In the above hypotheses (i), (ii) and (iii), we assume that there are no fixed points in $\Delta$, ruling out the possibility of a fixed point. So, the set consisting of endpoints of $C$ which are not in the boundary of $R$ is either empty or is a minimal period two orbit.
To prove the theorem we need only rule out the case of a minimal period 2 orbit. Case (i) is obvious. For the remaining cases, the proof proceeds along the same lines as the proof of Theorem 2 of \cite{kminvariantmanifolds}. 
The proof is nearly identical with the only exception being that Theorem 1 of this article replaces Theorem 1 of \cite{kminvariantmanifolds} everywhere Theorem 1 is mentioned or used. 
Such a replacement is valid because the proof of Theorem 2 of \cite{kminvariantmanifolds} does not make use of the assumption in Theorem 1 of \cite{kminvariantmanifolds} that $T$ is strongly competitive.
\end{proof}

\begin{thm}
 Under the hypotheses of Theorem 1, suppose there exists a neighborhood $U$ of $\bar{z}$ in $\mathbb{R}^{2}$ such that $T$ is of class $C^{k}$ on $U\cup \Delta$ for some $k\geq 1$, and that the Jacobian of $T$ at each $x\in\Delta$ is invertible. Then the curve $C$ in the conclusion of Theorem 1 is of class $C^{k}$.
\end{thm}
\begin{proof}
The proof proceeds along the same lines as the proof of Theorem 3 of \cite{kminvariantmanifolds}. 
The proof is nearly identical with the only exception being that Theorem 1 of this article replaces Theorem 1 of \cite{kminvariantmanifolds} everywhere Theorem 1 is mentioned or used. 
Such a replacement is valid because the proof of Theorem 3 of \cite{kminvariantmanifolds} does not make use of the assumption in Theorem 1 of \cite{kminvariantmanifolds} that $T$ is strongly competitive.
\end{proof}

\begin{thm}
\begin{enumerate}[(A)]
\item Assume the hypotheses of Theorem 1, and let $C$ be the curve whose existence is guaranteed by Theorem 1. If the endpoints of $C$ belong to $\partial R$, then $C$ separates $R$ into two connected components, namely
$$W_{-}:=\{x\in R\setminus C : \exists y\in C \;\; with \;\; x\preceq _{se} y \}$$
and
$$W_{+}:=\{x\in R\setminus C : \exists y\in C \;\; with \;\; y\preceq _{se} x \},$$
such that the following statements are true.
\begin{enumerate}[(i)]
\item $W_{-}$ is invariant, and $dist\left(T^{n}(x),Q_{2}(\bar{z})\right)\rightarrow 0$ as $n\rightarrow \infty$ for every $x\in W_{-}.$
\item $W_{+}$ is invariant, and $dist\left(T^{n}(x),Q_{4}(\bar{z})\right)\rightarrow 0$ as $n\rightarrow \infty$ for every $x\in W_{+}.$
\end{enumerate}
\item If, in addition to the hypotheses of part (A), $\bar{z}$ is an interior point of $R$, $T$ is $C^{2}$ in a neighborhood of $\bar{z}$, and one of the following holds:
\begin{enumerate}[(1)]
\item $\frac{\partial f}{\partial x}|_{(s,t)}>0$, $\frac{\partial f}{\partial y}|_{(s,t)}<0$, and $\frac{\partial g}{\partial x}|_{(s,t)}<0$ for all $(s,t)$ in a neighborhood of $\bar{z}$.
\item $\frac{\partial f}{\partial y}|_{(s,t)}<0$, $\frac{\partial g}{\partial x}|_{(s,t)}<0$, and $\frac{\partial g}{\partial y}|_{(s,t)}>0$ for all $(s,t)$ in a neighborhood of $\bar{z}$.
\end{enumerate}
Then $T$ has no periodic points in the boundary of $Q_{1}(\bar{z}) \cup Q_{3}(\bar{z})$ except for $\bar{z}$, and the following statements are true.

\begin{enumerate}[(iii)]
\item For every $x\in W_{-}$ there exists $n_{0}\in\mathbb{N}$ such that $T^{n}(x)\in int(Q_{2}(\bar{z}))$ for $n\geq n_{0}$.
\end{enumerate}

\begin{enumerate}[(iv)]
\item For every $x\in W_{+}$ there exists $n_{0}\in\mathbb{N}$ such that $T^{n}(x)\in int(Q_{4}(\bar{z}))$ for $n\geq n_{0}$.
\end{enumerate}

\end{enumerate}

\end{thm}

\begin{proof}
 The proof for (i) and (ii) proceeds along the same lines as the proof for (i) and (ii) of Theorem 4 of \cite{kminvariantmanifolds}. 
The proof is nearly identical with the only exception being that Theorem 1 of this article replaces Theorem 1 of \cite{kminvariantmanifolds} everywhere Theorem 1 is mentioned or used. 
Such a replacement is valid because the proof for (i) and (ii) of Theorem 4 of \cite{kminvariantmanifolds} does not make use of the assumption in Theorem 1 of \cite{kminvariantmanifolds} that $T$ is strongly competitive.\par

For the proof of (iii) and (iv), consider $T^{2}$. Notice that $T^{2}$ satisfies all of the necessary hypotheses needed to apply (iii) and (iv) of Theorem 4 of \cite{kminvariantmanifolds}, except that $T^{2}$ may not be strongly competitive on all of $\Delta$. 
The assumptions (1) and (2) were designed to force $T^{2}$ to be strongly competitive in a neighborhood of $\bar{z}$, allowing Theorem 4 of \cite{kminvariantmanifolds} to be used on $T^{2}$.
The proof for (iii) and (iv) of Theorem 4 of \cite{kminvariantmanifolds} does not make use of the assumption in Theorem 1 of \cite{kminvariantmanifolds} that $T$ is strongly competitive on $\Delta$, however Theorem 4 of \cite{kminvariantmanifolds} does need the map to be strongly competitive in a neighborhood of $\bar{z}$. So, we may apply Theorem 4 of \cite{kminvariantmanifolds} to the map $T^{2}$.\par
When we do so, we get that $T^{2}$ has no periodic points in the boundary of $Q_{1}(\bar{z}) \cup Q_{3}(\bar{z})$ except for $\bar{z}$. Thus, since every periodic point of $T$ is a periodic point of $T^{2}$, $T$ has no periodic points in the boundary of $Q_{1}(\bar{z}) \cup Q_{3}(\bar{z})$ except for $\bar{z}$.
We also get that for every $x\in W_{-}$ there exists $n_{0}\in\mathbb{N}$ such that $T^{2n}(x)\in int(Q_{2}(\bar{z}))$ for $n\geq n_{0}$. Since $int(Q_{2}(\bar{z}))$ is invariant for our map $T$ by virtue of $T$ being competitive and either (1) or (2), we get the following. For every $x\in W_{-}$, there exists $n_{0}\in\mathbb{N}$ such that $T^{n}(x)\in int(Q_{2}(\bar{z}))$ for $n\geq n_{0}$.
We also get that for every $x\in W_{+}$ there exists $n_{0}\in\mathbb{N}$ such that $T^{2n}(x)\in int(Q_{4}(\bar{z}))$ for $n\geq n_{0}$. Since $int(Q_{4}(\bar{z}))$ is invariant for our map $T$ by virtue of $T$ being competitive and either (1) or (2), we get the following. For every $x\in W_{+}$, there exists $n_{0}\in\mathbb{N}$ such that $T^{n}(x)\in int(Q_{4}(\bar{z}))$ for $n\geq n_{0}$.
\end{proof}

\section{The System (11,6)}
\begin{thm}
Consider the following system of rational difference equations
$$x_{n+1}=\frac{\alpha_{1}}{A_{1}+y_{n}},\quad y_{n+1}=\frac{\gamma_{2}y_{n}}{x_{n}},\quad n=0,1,\dots,$$
with $\alpha_1,A_{1},\gamma_{2}>0$ and arbitrary nonnegative initial conditions so that the denominator is never zero.
For this system of rational difference equations there are 3 regions in parametric space with distinct global behavior. The behavior is as follows:
\begin{enumerate}
\item If $\gamma_{2}>\frac{\alpha_{1}}{A_{1}}$, then the unique equilibrium $\left(\frac{\alpha_{1}}{A_{1}},0\right)$ is a saddle point with stable manifold $(0,\infty)\times \{0\}$. Whenever $(x_{0},y_{0})\not\in (0,\infty)\times \{0\}$, then $\lim_{n\rightarrow \infty}(x_{n},y_{n})=(0,\infty)$. 
\item If $\gamma_{2}=\frac{\alpha_{1}}{A_{1}}$, then the unique equilibrium $\left(\frac{\alpha_{1}}{A_{1}},0\right)$ is nonhyperbolic with basin of attraction $(0,\infty)\times \{0\}$. Whenever $(x_{0},y_{0})\not\in (0,\infty)\times \{0\}$, then $\lim_{n\rightarrow \infty}(x_{n},y_{n})=(0,\infty)$. 
\item If $\gamma_{2}<\frac{\alpha_{1}}{A_{1}}$, then there are exactly two equilibria $\left(\frac{\alpha_{1}}{A_{1}},0\right)$ and $\left(\gamma_{2},\frac{\alpha_{1}}{\gamma_{2}}-A_{1}\right)$. The equilibrium $\left(\frac{\alpha_{1}}{A_{1}},0\right)$ is locally asymptotically stable, and the equilibrium $\left(\gamma_{2},\frac{\alpha_{1}}{\gamma_{2}}-A_{1}\right)$ is a saddle point.
Moreover, there is a strictly increasing smooth function $C:(0,\infty)\rightarrow (0,\infty)$, which has the property that its graph $\{(x,y)\in (0,\infty)^{2}|y=C(x)\}$ is the global stable manifold for the saddle point equilibrium $\left(\gamma_{2},\frac{\alpha_{1}}{\gamma_{2}}-A_{1}\right)$ and an invariant separatrix. If $y_{0}>C(x_{0})$, then $\lim_{n\rightarrow \infty}(x_{n},y_{n})=(0,\infty)$. If $y_{0}=C(x_{0})$, then $\lim_{n\rightarrow \infty}(x_{n},y_{n})=\left(\gamma_{2},\frac{\alpha_{1}}{\gamma_{2}}-A_{1}\right)$. If $y_{0} < C(x_{0})$, then $\lim_{n\rightarrow \infty}(x_{n},y_{n})= \left(\frac{\alpha_{1}}{A_{1}},0\right)$. 
\end{enumerate}

\end{thm}
\begin{proof}
To find the equilibria we solve:
$$\bar{x}=\frac{\alpha_{1}}{A_{1}+\bar{y}},\quad \bar{y}=\frac{\gamma_{2}\bar{y}}{\bar{x}}.$$
So, 
$$(\bar{x}_{1},\bar{y}_{1})=\left(\frac{\alpha_{1}}{A_{1}}, 0 \right)$$
is an equilibrium. Furthermore, if $\gamma_{2}\in (0,\frac{\alpha_{1}}{A_{1}}]$, then
$$(\bar{x}_{2},\bar{y}_{2})=\left(\gamma_{2}, \frac{\alpha_{1}}{\gamma_{2}}-A_{1}\right)$$
is another equilibrium. Now let us prove that there are no other equilibria. From the equation
$$\bar{y}=\frac{\gamma_{2}\bar{y}}{\bar{x}},$$
we see that for any equilibrium point $\bar{y}=0$ or $\bar{x}=\gamma_{2}$. If the former is true, then the equilibrium is $(\bar{x}_{1},\bar{y}_{1})$. If the latter is true, then the equilibrium is $(\bar{x}_{2},\bar{y}_{2})$.
If $\gamma_{2}=\frac{\alpha_{1}}{A_{1}}$, then the two equilibria coincide and in fact there is a single equilibrium point.
Performing linearized stability analysis about the equilibrium $(\bar{x},\bar{y})$, we solve:
$$det\left(\begin{array}{cc}
-\lambda & \frac{-\bar{x}}{A_{1}+\bar{y}}\\               
\frac{-\bar{y}}{\bar{x}} & \frac{\gamma_{2}}{\bar{x}}-\lambda\\ 
\end{array}
\right)= 0.$$
$$\lambda^{2}-\left(\frac{\gamma_{2}}{\bar{x}}\right)\lambda - \frac{\bar{y}}{A_{1}+\bar{y}}=0.$$
So, the roots of the characteristic equation of the linearized equation about the equilibrium $(\bar{x}_{1},\bar{y}_{1})$ are the roots of
$$\lambda^{2}-\left(\frac{\gamma_{2}A_{1}}{\alpha_{1}}\right)\lambda =0.$$
So, the roots of the characteristic equation of the linearized equation about the equilibrium $(\bar{x}_{1},\bar{y}_{1})$ are $\lambda_{1}=0$ and $\lambda_{2}=\frac{\gamma_{2}A_{1}}{\alpha_{1}}$.
Moreover, the roots of the characteristic equation of the linearized equation about the equilibrium $(\bar{x}_{2},\bar{y}_{2})$ are the roots of
$$\lambda^{2}-\lambda + \frac{\gamma_{2}A_{1}}{\alpha_{1}}-1=0.$$
So, 
$$\lambda_{1,2}=\frac{1\pm \sqrt{5-\frac{4\gamma_{2}A_{1}}{\alpha_{1}}}}{2}.$$
Thus, if $\gamma_{2}>\frac{\alpha_{1}}{A_{1}}$, then the unique equilibrium 
$$(\bar{x},\bar{y})=\left(\frac{\alpha_{1}}{A_{1}}, 0 \right)$$
is a saddle point. If $\gamma_{2}=\frac{\alpha_{1}}{A_{1}}$, then the unique equilibrium 
$$(\bar{x},\bar{y})=\left(\frac{\alpha_{1}}{A_{1}}, 0 \right)$$
is nonhyperbolic. If $\gamma_{2}<\frac{\alpha_{1}}{A_{1}}$, then there are two equilibria,
$$(\bar{x}_{1},\bar{y}_{1})=\left(\frac{\alpha_{1}}{A_{1}}, 0 \right)\quad and \quad (\bar{x}_{2},\bar{y}_{2})=\left(\gamma_{2}, \frac{\alpha_{1}}{\gamma_{2}}-A_{1}\right).$$
Furthermore, in this region of the parameters $\frac{\gamma_{2}A_{1}}{\alpha_{1}} < 1$ and 
$$\frac{1 + \sqrt{5-\frac{4\gamma_{2}A_{1}}{\alpha_{1}}}}{2} > \frac{1 + \sqrt{5-\frac{4\alpha_{1}A_{1}}{A_{1}\alpha_{1}}}}{2} = 1.$$
Moreover, in this region of the parameters 
$$\frac{1 - \sqrt{5-\frac{4\gamma_{2}A_{1}}{\alpha_{1}}}}{2} > \frac{1 - \sqrt{5}}{2} > -1.$$
So, if $\gamma_{2}<\frac{\alpha_{1}}{A_{1}}$, then $(\bar{x}_{1},\bar{y}_{1})$ is locally asymptotically stable and $(\bar{x}_{2},\bar{y}_{2})$ is a saddle point.
Consider the case where $\gamma_{2}>\frac{\alpha_{1}}{A_{1}}$. In this case, we have 
$$y_{n+2}=\frac{\gamma_{2}y_{n+1}}{x_{n+1}}\geq \frac{\gamma_{2}A_{1}y_{n+1}}{\alpha_{1}},\quad n\geq 0.$$
So, in this case, whenever $(x_{0},y_{0})\not\in (0,\infty)\times \{0\}$, then $\lim_{n\rightarrow \infty}(x_{n},y_{n})=(0,\infty)$. If $\gamma_{2}>\frac{\alpha_{1}}{A_{1}}$ and $(x_{0},y_{0})\in (0,\infty)\times \{0\}$, then $(x_{n},y_{n})=(\frac{\alpha_{1}}{A_{1}},0)$ for $n\geq 1$. Thus, in this case, $(0,\infty)\times \{0\}$ is the stable manifold for the saddle point equilibrium $(\frac{\alpha_{1}}{A_{1}},0)$.
Consider the case where $\gamma_{2}=\frac{\alpha_{1}}{A_{1}}$. In this case, we have 
$$y_{n+2}=\frac{\gamma_{2}y_{n+1}}{x_{n+1}}= y_{n+1}+\frac{\gamma_{2}y_{n+1}y_{n}}{\alpha_{1}},\quad n\geq 0.$$
So, the solution $\{(x_{n},y_{n})\}^{\infty}_{n=0}$ has $\{y_{n}\}^{\infty}_{n=1}$ monotone increasing. Furthermore,
$$y_{n+2}=\frac{\gamma_{2}y_{n+1}}{x_{n+1}}= y_{n+1}+\frac{\gamma_{2}y_{n+1}y_{n}}{\alpha_{1}}\geq y_{n+1}+\frac{\gamma_{2}\min(y_{0},y_{1})^{2}}{\alpha_{1}},\quad n\geq 0.$$
So, in this case, whenever $(x_{0},y_{0})\not\in (0,\infty)\times \{0\}$, then $\lim_{n\rightarrow \infty}(x_{n},y_{n})=(0,\infty)$. If $\gamma_{2}=\frac{\alpha_{1}}{A_{1}}$ and $(x_{0},y_{0})\in (0,\infty)\times \{0\}$, then $(x_{n},y_{n})=(\frac{\alpha_{1}}{A_{1}},0)$ for $n\geq 1$. Thus, in this case, $(0,\infty)\times \{0\}$ is an invariant manifold which is the basin of attraction for the nonhyperbolic equilibrium $(\frac{\alpha_{1}}{A_{1}},0)$.
Now, consider the case $\gamma_{2}<\frac{\alpha_{1}}{A_{1}}$. For this case we want to apply Theorems 1, 2, 3, and 4, so we check the hypotheses of these theorems.
Consider the region $R=(0, \infty)^{2}$. Let us consider the competitive map 
$$T\left(\begin{array}{cc}
          x\\
y\\
         \end{array}
\right)= \left(\begin{array}{cc}
f(x,y)\\
g(x,y)\\
         \end{array}
\right)= \left(\begin{array}{cc}
\frac{\alpha_{1}}{A_{1}+y}\\
\frac{\gamma_{2}y}{x}\\
         \end{array}
\right).$$
Call the fixed point $\left(\gamma_{2}, \frac{\alpha_{1}}{\gamma_{2}}-A_{1}\right)$, $\bar{z}$. This fixed point will act as the fixed point, $\bar{z}$, in the statement of Theorem 1.
Notice that $\Delta := R\cap int\left(Q_{1}(\bar{z})\cup Q_{3}(\bar{z})\right)$ is nonempty. Also, notice that
$$\frac{\partial f}{\partial y}|_{(s,t)}= \frac{-\alpha_{1}}{(A_{1}+t)^{2}}<0\;\;for\;\;all\;\; (s,t)\in\Delta.$$ 
$$\frac{\partial g}{\partial x}|_{(s,t)}=\frac{-\gamma_{2}t}{s^{2}}<0\;\;for\;\;all\;\; (s,t)\in\Delta.$$
Moreover, $T$ is $C^{\infty}$ on all of $R$. Also, $J_{T}(\bar{z})$ has eigenvalues,
$$\mu= \frac{1 + \sqrt{5-\frac{4\gamma_{2}A_{1}}{\alpha_{1}}}}{2},$$
$$\lambda= \frac{1 - \sqrt{5-\frac{4\gamma_{2}A_{1}}{\alpha_{1}}}}{2}.$$
So, as we have shown earlier,
$$-1< \frac{1 - \sqrt{5-\frac{4\gamma_{2}A_{1}}{\alpha_{1}}}}{2} < 0 < 1 < \frac{1 + \sqrt{5-\frac{4\gamma_{2}A_{1}}{\alpha_{1}}}}{2}.$$
Further, notice that
$$J_{T}(\bar{z})=\left(\begin{array}{cc}
0 & \frac{-\gamma^{2}_{2}}{\alpha_{1}}\\               
\frac{A_{1}}{\gamma_{2}}-\frac{\alpha_{1}}{\gamma^{2}_{2}} & 1\\ 
\end{array}
\right).$$
Thus, the eigenspace $E^{\lambda}$ is not a coordinate axis. Moreover, the value of $J_{T}(\bar{z})$ and the fact that $T$ is $C^{\infty}$ on all of $R$ tells us that the hypothesis (2) needed for Theorem 4 holds.\par
Notice that since $\gamma_{2}<\frac{\alpha_{1}}{A_{1}}$, the equilibrium $(\bar{x}_{1},\bar{y}_{1})\not\in\Delta$. Also notice that 
$$det(J_{T}(\bar{z}))=det\left(\begin{array}{cc}
0 & \frac{-\gamma^{2}_{2}}{\alpha_{1}}\\               
\frac{A_{1}}{\gamma_{2}}-\frac{\alpha_{1}}{\gamma^{2}_{2}} & 1\\ 
\end{array}
\right)= - 1 + \frac{\gamma_{2}A_{1}}{\alpha_{1}}<0 .$$
Furthermore, the following system of equations has a unique solution,
$$\gamma_{2}=\frac{\alpha_{1}}{A_{1}+y},$$
$$A_{1}-\frac{\alpha_{1}}{\gamma_{2}}=\frac{\gamma_{2}y}{x}.$$
Thus, Theorems 1, 2, 3, and 4 apply on the region $R=(0, \infty)^{2}$. Therefore, the consequences of Theorems 1, 2, 3, and 4 are true on the region $(0, \infty)^{2}$.
Thus, there is a strictly increasing $C^{\infty}$ function on $(0,\infty)$, $C(x)$, whose graph passes through the point $\left(\gamma_{2}, \frac{\alpha_{1}}{\gamma_{2}}-A_{1}\right)$. The graph of $C(x)$ separates $(0,\infty)^{2}$ into two invariant regions,
$$W_{-}:= \{(x,y)\in (0,\infty)^{2}| y>C(x)\}.$$
$$W_{+}:=\{(x,y)\in (0,\infty)^{2}| y<C(x)\}.$$
Theorem 4(B) then tells us that for each point $(x_{0},y_{0})\in W_{-}$, there exists an $n_{0}\in\mathbb{N}$ such that $x_{n}< \gamma_{2}$ and $y_{n}> \frac{\alpha_{1}}{\gamma_{2}}-A_{1}$ for $n\geq n_{0}$.
We claim that for sufficiently small $\epsilon >0$ sets of the form $(0,\gamma_{2}-\epsilon]\times [\frac{\alpha_{1}}{\gamma_{2}-\epsilon}-A_{1}, \infty)$ are invariant. Let us prove this claim.
Suppose $(x_{n},y_{n})\in (0,\gamma_{2}-\epsilon]\times [\frac{\alpha_{1}}{\gamma_{2}-\epsilon}-A_{1}, \infty)$, then 
$$y_{n+1}=\frac{\gamma_{2}y_{n}}{x_{n}}\geq y_{n} \geq \frac{\alpha_{1}}{\gamma_{2}-\epsilon}-A_{1}.$$
Moreover,
$$x_{n+1}=\frac{\alpha_{1}}{A_{1}+y_{n}}\leq \frac{\alpha_{1}}{A_{1}+\frac{\alpha_{1}}{\gamma_{2}-\epsilon}-A_{1}}= \gamma_{2}-\epsilon.$$
So, for some $\epsilon>0$, $y_{n+1} > \frac{\gamma_{2}}{\gamma_{2}-\epsilon }y_{n} $ for $n> n_{0}$.
Thus, by this fact and our original recursive equation, for each initial condition $(x_{0},y_{0})\in W_{-}$, $\lim_{n\rightarrow \infty}(x_{n},y_{n})=(0,\infty)$. \par
Further, Theorem 4(B) tells us that for each point $(x_{0},y_{0})\in W_{+}$, there exists an $n_{0}\in\mathbb{N}$ such that $x_{n}> \gamma_{2}$ and $y_{n}< \frac{\alpha_{1}}{\gamma_{2}}-A_{1}$ for $n\geq n_{0}$.
We claim that for sufficiently small $\epsilon >0$ sets of the form $[\gamma_{2}+\epsilon,\infty)\times [0,\frac{\alpha_{1}}{\gamma_{2}+\epsilon}-A_{1}]$ are invariant. Let us prove this claim.
Suppose $(x_{n},y_{n})\in [\gamma_{2}+\epsilon,\infty)\times [0,\frac{\alpha_{1}}{\gamma_{2}+\epsilon}-A_{1}]$, then 
$$y_{n+1}=\frac{\gamma_{2}y_{n}}{x_{n}}\leq y_{n} \leq \frac{\alpha_{1}}{\gamma_{2}+\epsilon}-A_{1}.$$
Moreover,
$$x_{n+1}=\frac{\alpha_{1}}{A_{1}+y_{n}}\geq \frac{\alpha_{1}}{A_{1}+\frac{\alpha_{1}}{\gamma_{2}+\epsilon}-A_{1}}= \gamma_{2}+\epsilon.$$
So, for some $\epsilon>0$, $y_{n+1} < \frac{\gamma_{2}}{\gamma_{2}+\epsilon }y_{n} $ for $n> n_{0}$.
Thus, by this fact and our original recursive equation, for each initial condition $(x_{0},y_{0})\in W_{+}$, $\lim_{n\rightarrow \infty}(x_{n},y_{n})=\left(\frac{\alpha_{1}}{A_{1}}, 0 \right)$. 
Moreover, by Theorems 1, 2, 3, and 4, the set 
$$\{(x,y)\in (0,\infty)^{2}| y=C(x)\}$$
is an invariant separatrix. Also, for $(x_{0},y_{0})\in \{(x,y)\in (0,\infty)^{2}| y=C(x)\}$, $\lim_{n\rightarrow \infty}(x_{n},y_{n})=\left(\gamma_{2}, \frac{\alpha_{1}}{\gamma_{2}}-A_{1}\right)$.
Suppose $(x_{0},y_{0})\in (0,\infty )\times \{0\}$, then $(x_{n},y_{n})= \left(\frac{\alpha_{1}}{A_{1}}, 0 \right)$ for $n\geq 1$. Suppose $(x_{0},y_{0})\in  \{0\}\times [0,\infty )$, then $y_{1}$ is undefined.

\end{proof}

\section{The System (11,14)}
\begin{thm}
Consider the following system of rational difference equations
$$x_{n+1}=\frac{\alpha_{1}}{A_{1}+y_{n}},\quad y_{n+1}=\frac{y_{n}}{A_{2}+x_{n}},\quad n=0,1,\dots,$$
with $\alpha_1,A_{1},A_{2}>0$ and arbitrary nonnegative initial conditions so that the denominator is never zero.
For this system of rational difference equations there are 4 regions in parametric space with distinct global behavior. The behavior is as follows:
\begin{enumerate}
\item If $A_{2}\geq 1$, then the unique equilibrium $\left(\frac{\alpha_{1}}{A_{1}},0\right)$ is globally asymptotically stable.
\item If $A_{2}+\frac{\alpha_{1}}{A_{1}} < 1$, then the unique equilibrium $\left(\frac{\alpha_{1}}{A_{1}},0\right)$ is a saddle point with stable manifold $[0,\infty)\times \{0\}$. Whenever $(x_{0},y_{0})\not\in [0,\infty)\times \{0\}$, then $\lim_{n\rightarrow \infty}(x_{n},y_{n})=(0,\infty)$. 
\item If $A_{2}+\frac{\alpha_{1}}{A_{1}} = 1$, then the unique equilibrium $\left(\frac{\alpha_{1}}{A_{1}},0\right)$ is nonhyperbolic with basin of attraction $[0,\infty)\times \{0\}$. Whenever $(x_{0},y_{0})\not\in [0,\infty)\times \{0\}$, then $\lim_{n\rightarrow \infty}(x_{n},y_{n})=(0,\infty)$. 
\item If $A_{2}+\frac{\alpha_{1}}{A_{1}} > 1$, yet $A_{2} < 1$, then there are exactly two equilibria $\left(\frac{\alpha_{1}}{A_{1}},0\right)$ and $\left(1-A_{2},\frac{\alpha_{1}}{1-A_{2}}-A_{1}\right)$. The equilibrium $\left(\frac{\alpha_{1}}{A_{1}},0\right)$ is locally asymptotically stable, and the equilibrium $\left(1-A_{2},\frac{\alpha_{1}}{1-A_{2}}-A_{1}\right)$ is a saddle point.
Moreover, there is a strictly increasing smooth function $C:(0,\infty)\rightarrow (0,\infty)$, which has the property that its graph $\{(x,y)\in (0,\infty)^{2}|y=C(x)\}$ is the stable manifold for the saddle point equilibrium $\left(1-A_{2},\frac{\alpha_{1}}{1-A_{2}}-A_{1}\right)$ and an invariant separatrix for all positive solutions. Whenever $x_{0}>0$, the behavior of solutions can be described as follows. If $y_{0}>C(x_{0})$, then $\lim_{n\rightarrow \infty}(x_{n},y_{n})=(0,\infty)$. If $y_{0}=C(x_{0})$, then $\lim_{n\rightarrow \infty}(x_{n},y_{n})=\left(1-A_{2},\frac{\alpha_{1}}{1-A_{2}}-A_{1}\right)$. If $y_{0} < C(x_{0})$, then $\lim_{n\rightarrow \infty}(x_{n},y_{n})= \left(\frac{\alpha_{1}}{A_{1}},0\right)$.
Whenever $x_{0}=0$, we substitute $(x_{1},y_{1})$ for $(x_{0},y_{0})$ in the prior statements to determine the global behavior. 
\end{enumerate}

\end{thm}
\begin{proof}
To find the equilibria we solve:
$$\bar{x}=\frac{\alpha_{1}}{A_{1}+\bar{y}},\quad \bar{y}=\frac{\bar{y}}{A_{2}+\bar{x}}.$$
So, $(\bar{x}_{1},\bar{y}_{1})=(\frac{\alpha_{1}}{A_{1}},0)$ and if $A_{2}<1$, $\alpha_{1}> A_{1}-A_{1}A_{2}$, and $\bar{y}\neq 0$, then 
$$(\bar{x}_{2},\bar{y}_{2})=(1-A_{2},\frac{\alpha_{1}}{1-A_{2}}-A_{1}).$$
Performing linearized stability analysis about the equilibrium $(\bar{x},\bar{y})$ we solve:
$$det\left(\begin{array}{cc}
-\lambda & \frac{-\bar{x}}{A_{1}+\bar{y}}\\               
\frac{-\bar{y}}{A_{2}+\bar{x}} & \frac{1}{A_{2}+\bar{x}}-\lambda\\ 
\end{array}
\right)= 0.$$
$$\lambda^{2}-\left(\frac{1}{A_{2}+\bar{x}}\right)\lambda + \frac{-\bar{x}\bar{y}}{\left(A_{1}+\bar{y}\right)\left(A_{2}+\bar{x}\right)}=0.$$
Thus, for $(\bar{x}_{1},\bar{y}_{1})=(\frac{\alpha_{1}}{A_{1}},0)$, we get
$$\lambda^{2}-\left(\frac{1}{A_{2}+\frac{\alpha_{1}}{A_{1}}}\right)\lambda =0.$$
So, $\lambda_{11}=0$ and $\lambda_{12}=\frac{1}{A_{2}+\frac{\alpha_{1}}{A_{1}}}$.
For $(\bar{x}_{2},\bar{y}_{2})=(1-A_{2},\frac{\alpha_{1}}{1-A_{2}}-A_{1})$, we get
$$\lambda^{2}-\lambda - \frac{(\alpha_{1}-A_{1}+A_{1}A_{2})}{\frac{\alpha_{1}}{1-A_{2}}}=0,$$
$$\lambda^{2}-\lambda +A_{2}-1 +\frac{A_{1}(1-A_{2})^{2}}{\alpha_{1}}=0.$$
So,
$$\lambda_{21,22}=\frac{1\pm\sqrt{5-4A_{2}-4\frac{A_{1}(1-A_{2})^{2}}{\alpha_{1}}}}{2}.$$
Thus, if $A_{2}+\frac{\alpha_{1}}{A_{1}}<1$, then $(\frac{\alpha_{1}}{A_{1}},0)$ is the unique nonnegative equilibrium and it is a saddle point.
If $A_{2}+\frac{\alpha_{1}}{A_{1}}=1$, then $(\frac{\alpha_{1}}{A_{1}},0)$ is the unique nonnegative equilibrium and it is nonhyperbolic.
If $A_{2}+\frac{\alpha_{1}}{A_{1}}>1$ and $A_{2}<1$, then  $A_{2}+\frac{A_{1}(1-A_{2})^{2}}{\alpha_{1}}< A_{2}+\frac{A_{1}(1-A_{2})\left(\frac{\alpha_{1}}{A_{1}}\right)}{\alpha_{1}}= 1$, so the equilibrium $(\frac{\alpha_{1}}{A_{1}},0)$ is locally asymptotically stable and the equilibrium $(1-A_{2},\frac{\alpha_{1}}{1-A_{2}}-A_{1})$ is a saddle point.
If $A_{2}\geq 1$, then $(\frac{\alpha_{1}}{A_{1}},0)$ is the unique nonnegative equilibrium and it is locally asymptotically stable.\par
Suppose $A_{2}+\frac{\alpha_{1}}{A_{1}}<1$ and $(x_{0},y_{0})\not\in [0,\infty)\times \{0\}$, then $x_{1}< \frac{\alpha_{1}}{A_{1}}$, so $y_{n}> \frac{y_{n-1}}{A_{2}+ \frac{\alpha_{1}}{A_{1}}} $ for $n\geq 2$. Thus, since $y_{0}\neq 0$ and $A_{2}+\frac{\alpha_{1}}{A_{1}}<1$, $\lim_{n\rightarrow\infty} (x_{n},y_{n})=(0,\infty)$. 
Thus, $[0,\infty)\times (0,\infty)$ is a basin of attraction for $(0,\infty)$ in this case. 
Suppose $A_{2}+\frac{\alpha_{1}}{A_{1}}<1$ and $(x_{0},y_{0})\in [0,\infty)\times \{0\}$, then $(x_{n},y_{n})= (\frac{\alpha_{1}}{A_{1}},0)$ for $n\geq 1$. Thus, $[0,\infty)\times \{0\}$ is the stable manifold for the saddle equilibrium $(\frac{\alpha_{1}}{A_{1}},0)$ in this case.\par
Suppose $A_{2}+\frac{\alpha_{1}}{A_{1}}=1$ and $(x_{0},y_{0})\not\in [0,\infty)\times \{0\}$, then $x_{1}< \frac{\alpha_{1}}{A_{1}}$, so 
$$y_{n}> \frac{y_{n-1}}{A_{2}+ \frac{\alpha_{1}}{A_{1}}} = y_{n-1} \;\; for\;\; n\geq 2.$$
Thus, $y_{n}$ is an eventually monotone increasing sequence of real numbers, and so either converges to a finite limit or diverges to $\infty$.
 If $y_{n}$ converges to a finite limit, then $y_{n}$ must converge to the second coordinate of an equilibrium point. This follows from the continuity of our map. 
However, the only equilibrium in this case is the nonhyperbolic equilibrium $(\frac{\alpha_{1}}{A_{1}},0)$, with second coordinate $0$. Therefore, since $\{y_{n}\}^{\infty}_{n=2}$ is a positive monotone increasing sequence, it cannot converge to $0$. Thus, $y_{n}$ diverges monotonically to $\infty$.
So, if $A_{2}+\frac{\alpha_{1}}{A_{1}}=1$ and $(x_{0},y_{0})\not\in [0,\infty)\times \{0\}$, then $\lim_{n\rightarrow\infty} (x_{n},y_{n})=(0,\infty)$. 
Thus, $[0,\infty)\times (0,\infty)$ is a basin of attraction for $(0,\infty)$ in this case. 
Suppose $A_{2}+\frac{\alpha_{1}}{A_{1}}=1$ and $(x_{0},y_{0})\in [0,\infty)\times \{0\}$, then $(x_{n},y_{n})= (\frac{\alpha_{1}}{A_{1}},0)$ for $n\geq 1$. Thus, $[0,\infty)\times \{0\}$ is an invariant manifold which is the basin of attraction for the nonhyperbolic equilibrium $(\frac{\alpha_{1}}{A_{1}},0)$ in this case.\par
Suppose $A_{2}\geq 1$, then 
$$y_{n}\leq y_{n-1} \;\; for\;\; n\geq 1.$$
Hence, $y_{n}$ is a monotone decreasing sequence of real numbers which is bounded below by zero. Thus, $y_{n}$ converges to a finite limit and $y_{n}$ must converge to the second coordinate of an equilibrium point. This follows from the continuity of our map.
Thus, $y_{n}$ converges monotonically to zero. So, in this case, $(\frac{\alpha_{1}}{A_{1}},0)$ is a locally asymptotically stable global attractor and so globally asymptotically stable.\par
Now, consider the case $A_{2}+\frac{\alpha_{1}}{A_{1}}>1$ and $A_{2}<1$. For this case we want to apply Theorems 1, 2, 3, and 4, so we check the hypotheses of these theorems. 
Consider the region $R=(0 ,\infty)^{2}$. Let us consider the competitive map
$$T\left(\begin{array}{cc}
          x\\
y\\
         \end{array}
\right)= \left(\begin{array}{cc}
f(x,y)\\
g(x,y)\\
         \end{array}
\right)= \left(\begin{array}{cc}
\frac{\alpha_{1}}{A_{1}+y}\\
\frac{y}{A_{2}+x}\\
         \end{array}
\right).$$
Call the fixed point $\left(1-A_{2},\frac{\alpha_{1}}{1-A_{2}}-A_{1}\right)$, $\bar{z}$. This fixed point will act as the fixed point, $\bar{z}$, in the statement of Theorem 1.
Notice that $\Delta := R\cap int(Q_{1}(\bar{z})\cup Q_{3}(\bar{z}))$ is nonempty. Also, notice that
$$\frac{\partial f}{\partial y}|_{(s,t)}= \frac{-\alpha_{1}}{(A_{1}+t)^{2}}<0\;\;for\;\;all\;\; (s,t)\in\Delta.$$ 
$$\frac{\partial g}{\partial x}|_{(s,t)}=\frac{-t}{(A_{2}+s)^{2}}<0\;\;for\;\;all\;\; (s,t)\in\Delta.$$
Moreover, $T$ is $C^{\infty}$ on all of $R$. Also, $J_{T}(\bar{z})$ has eigenvalues
$$\mu= \frac{1+\sqrt{5-4A_{2}-4\frac{A_{1}(1-A_{2})^{2}}{\alpha_{1}}}}{2},$$
$$\lambda= \frac{1-\sqrt{5-4A_{2}-4\frac{A_{1}(1-A_{2})^{2}}{\alpha_{1}}}}{2}.$$
So, from the inequality $A_{2}+\frac{A_{1}(1-A_{2})^{2}}{\alpha_{1}}< 1$ which holds in this region of the parameters, we get
$$-1< \frac{1-\sqrt{5-4A_{2}-4\frac{A_{1}(1-A_{2})^{2}}{\alpha_{1}}}}{2} < 0 < 1 <\frac{1+\sqrt{5-4A_{2}-4\frac{A_{1}(1-A_{2})^{2}}{\alpha_{1}}}}{2}.$$
Further, notice that
$$J_{T}(\bar{z})=\left(\begin{array}{cc}
0 & \frac{-(1-A_{2})^{2}}{\alpha_{1}}\\               
A_{1}-\frac{\alpha_{1}}{1-A_{2}} & 1\\ 
\end{array}
\right).$$
Thus, since $A_{2}+\frac{\alpha_{1}}{A_{1}}>1$ and $A_{2}<1$, $A_{1}-\frac{\alpha_{1}}{1-A_{2}}<0$. So, the eigenspace $E^{\lambda}$ is not a coordinate axis. Moreover, the value of $J_{T}(\bar{z})$ and the fact that $T$ is $C^{\infty}$ on all of $R$ tells us that the hypothesis (2) needed for Theorem 4 holds.\par
Notice that since $A_{2}+\frac{\alpha_{1}}{A_{1}}>1$, $1-A_{2}< \frac{\alpha_{1}}{A_{1}}$ and so, the equilibrium $(\bar{x}_{1},\bar{y}_{1})\not\in\Delta$. Also notice that
$$det(J_{T}(\bar{z}))=det\left(\begin{array}{cc}
0 & \frac{-(1-A_{2})^{2}}{\alpha_{1}}\\               
A_{1}-\frac{\alpha_{1}}{1-A_{2}} & 1\\ 
\end{array}
\right)= A_{2} - 1 + \frac{(1-A_{2})^{2}A_{1}}{\alpha_{1}}<0 ,$$
from the inequality $A_{2}+\frac{A_{1}(1-A_{2})^{2}}{\alpha_{1}}< 1$ which holds in this region of the parameters. 
Furthermore, the following system of equations has a unique solution,
$$1-A_{2}=\frac{\alpha_{1}}{A_{1}+y},$$
$$\frac{\alpha_{1}}{1-A_{2}}-A_{1}=\frac{y}{A_{2}+x}.$$
Thus, Theorems 1, 2, 3, and 4 apply on the region $R=(0, \infty)^{2}$. Therefore, the consequences of Theorems 1, 2, 3, and 4 are true on the region $(0, \infty)^{2}$.
Thus, there is a strictly increasing $C^{\infty}$ function on $(0,\infty)$, $C(x)$, whose graph passes through the point $\left(1-A_{2},\frac{\alpha_{1}}{1-A_{2}}-A_{1}\right)$. The graph of $C(x)$ separates $(0,\infty)^{2}$ into two invariant regions,
$$W_{-}:= \{(x,y)\in (0,\infty)^{2}| y>C(x)\}.$$
$$W_{+}:=\{(x,y)\in (0,\infty)^{2}| y<C(x)\}.$$
Theorem 4(B) then tells us that for each point $(x_{0},y_{0})\in W_{-}$, there exists an $n_{0}\in\mathbb{N}$ such that $x_{n}< 1-A_{2}$ and $y_{n}> \frac{\alpha_{1}}{1-A_{2}}-A_{1}$ for $n\geq n_{0}$.
We claim that for sufficiently small $\epsilon >0$ sets of the form $[0,1-A_{2}-\epsilon]\times [\frac{\alpha_{1}}{1-A_{2}-\epsilon}-A_{1}, \infty)$ are invariant. Let us prove this claim.
Suppose $(x_{n},y_{n})\in [0,1-A_{2}-\epsilon]\times [\frac{\alpha_{1}}{1-A_{2}-\epsilon}-A_{1}, \infty)$, then 
$$y_{n+1}=\frac{y_{n}}{A_{2}+x_{n}}\geq y_{n} \geq \frac{\alpha_{1}}{1-A_{2}-\epsilon}-A_{1}.$$
Moreover,
$$x_{n+1}=\frac{\alpha_{1}}{A_{1}+y_{n}}\leq \frac{\alpha_{1}}{A_{1}+\frac{\alpha_{1}}{1-A_{2}-\epsilon}-A_{1}}= 1-A_{2}-\epsilon.$$
So, for some $\epsilon>0$, $y_{n+1} > \frac{1}{A_{2}+1-A_{2}-\epsilon }y_{n} $ for $n> n_{0}$.
Thus, by this fact and our original recursive equation, for each initial condition $(x_{0},y_{0})\in W_{-}$, $\lim_{n\rightarrow \infty}(x_{n},y_{n})=(0,\infty)$. \par
Further, Theorem 4(B) tells us that for each point $(x_{0},y_{0})\in W_{+}$, there exists an $n_{0}\in\mathbb{N}$ such that $x_{n}> 1-A_{2}$ and $y_{n}< \frac{\alpha_{1}}{1-A_{2}}-A_{1}$ for $n\geq n_{0}$.
We claim that for sufficiently small $\epsilon >0$ sets of the form $[1-A_{2}+\epsilon,\infty)\times [0,\frac{\alpha_{1}}{1-A_{2}+\epsilon}-A_{1}]$ are invariant. Let us prove this claim.
Suppose $(x_{n},y_{n})\in [1-A_{2}+\epsilon,\infty)\times [0,\frac{\alpha_{1}}{1-A_{2}+\epsilon}-A_{1}]$, then 
$$y_{n+1}=\frac{y_{n}}{A_{2}+x_{n}}\leq y_{n} \leq \frac{\alpha_{1}}{1-A_{2}+\epsilon}-A_{1}.$$
Moreover,
$$x_{n+1}=\frac{\alpha_{1}}{A_{1}+y_{n}}\geq \frac{\alpha_{1}}{A_{1}+\frac{\alpha_{1}}{1-A_{2}+\epsilon}-A_{1}}= 1-A_{2}+\epsilon.$$
So, for some $\epsilon>0$, $y_{n+1} < \frac{1}{A_{2}+1-A_{2}+\epsilon }y_{n} $ for $n> n_{0}$.
Thus, by this fact and our original recursive equation, for each initial condition $(x_{0},y_{0})\in W_{+}$, $\lim_{n\rightarrow \infty}(x_{n},y_{n})=\left(\frac{\alpha_{1}}{A_{1}}, 0 \right)$. \par
We also get via Theorems 1, 2, 3, and 4 that the set 
$$\{(x,y)\in (0,\infty)^{2}| y=C(x)\}$$
is an invariant separatrix. Moreover, for $(x_{0},y_{0})\in \{(x,y)\in (0,\infty)^{2}| y=C(x)\}$, $\lim_{n\rightarrow \infty}(x_{n},y_{n})=\left(1-A_{2},\frac{\alpha_{1}}{1-A_{2}}-A_{1}\right)$.\par
Suppose $(x_{0},y_{0})\in (0,\infty )\times \{0\}$, then $(x_{n},y_{n})= \left(\frac{\alpha_{1}}{A_{1}}, 0 \right)$ for $n\geq 1$. Suppose $(x_{0},y_{0})\in  \{0\}\times [0,\infty )$, then $(x_{1},y_{1})\in (0,\infty)^{2}$ and we may apply the prior results.
\end{proof}

\section{Conclusion}
This paper is the first in a series of papers which will address the nontrivial competitive special cases of the following rational system in the plane, labeled system \#11. 
$$x_{n+1}=\frac{\alpha_{1}}{A_{1}+y_{n}},\quad y_{n+1}=\frac{\alpha_{2}+\beta_{2}x_{n}+\gamma_{2}y_{n}}{A_{2}+B_{2}x_{n}+C_{2}y_{n}},\quad n=0,1,2,\dots ,$$
with $\alpha_{1},A_{1}>0$ and $\alpha_{2}, \beta_{2}, \gamma_{2}, A_{2}, B_{2}, C_{2}\geq 0$ so that $\alpha_{2}+\beta_{2}+\gamma_{2}>0$ and $A_{2}+B_{2}+C_{2}>0$ and with nonnegative initial conditions $x_{0}$ and $y_{0}$ so that the denominator is never zero.
The nontrivial competitive special cases of system \#11 are the cases numbered $(11,6)$, $(11,14)$, $(11,15)$, $(11,21)$, $(11,27)$, $(11,29)$, $(11,38)$, and $(11,42)$, in the numbering system developed in \cite{cklm}.\par  
In this article, we have generalized several powerful results of Kulenovi\'c and Merino and used the new results to determine a near complete picture of the qualitative behavior for the two rational systems in the plane $(11,6)$ and $(11,14)$. Further work will focus on obtaining a near complete picture of the qualitative behavior for the remaining cases, namely  $(11,15)$, $(11,21)$, $(11,27)$, $(11,29)$, $(11,38)$, and $(11,42)$.\par
Theorems 1, 2, 3, and 4 of this article may be of use to mathematical biologists as well as other researchers studying rational systems in the plane. The preliminary and introductory material describing the standard results on order preserving and competitive maps was largely based off of the introductory material in \cite{kmbifurcation} and \cite{kminvariantmanifolds}.
\par
\end{document}